\documentclass{amsart}

\usepackage[english]{babel}

\usepackage{amsmath, amsfonts, amssymb, amsthm, faktor}

\usepackage[all]{xy}
\usepackage{tikz}
\usetikzlibrary{math}
\usetikzlibrary{patterns}
\usepackage{caption}
\usepackage{subfig}

\usepackage{caption}
\usetikzlibrary{arrows,chains,matrix,positioning,scopes,decorations.pathreplacing,
decorations.pathmorphing,decorations.markings,arrows.meta}

\usepackage{aliascnt}
\usepackage[colorlinks, linktocpage, allcolors=black,breaklinks]{hyperref}

\usepackage{enumerate}

\usepackage{verbatim}

\theoremstyle{definition}
\newtheorem{theorem}{Theorem}[section]

\newtheorem{corollary}[theorem]{Corollary}
\newtheorem{lemma}[theorem]{Lemma}

\begin{document}

\title[Strong Marker Sets and Applications]{Strong Marker Sets and Applications}

\author{Su Gao}
\address{School of Mathematical Sciences and LPMC, Nankai University, Tianjin 300071, P.R. China}
\email{sgao@nankai.edu.cn}
\thanks{The authors acknowledge the partial support of their research by the National Natural Science Foundation of China (NSFC) grants 12250710128 and 12271263.}

\author{Tianhao Wang}
\address{School of Mathematical Sciences and LPMC, Nankai University, Tianjin 300071, P.R. China}
\email{tianhao\_wang@qq.com}

\date{\today}

\begin{abstract} We prove the existence of clopen marker sets with some strong regularity property. For each $n\geq 1$ and any integer $d\geq 1$, we show that there are a positive integer $D$ and a clopen marker set $M$ in $F(2^{\mathbb{Z}^n})$ such that
\begin{enumerate}
\item[(1)] for any distinct $x,y\in M$ in the same orbit, $\rho(x,y)\geq d$;
\item[(2)] for any $1\leq i\leq n$ and any $x\in F(2^{\mathbb{Z}^n})$, there are non-negative integers $a, b\leq D$ such that $a\cdot x\in M$ and $-b\cdot x\in M$.
\end{enumerate}
As an application, we obtain a clopen tree section for $F(2^{\mathbb{Z}^n})$. Based on the strong marker sets, we get a quick proof that there exist clopen continuous edge $(2n+1)$-colorings of $F(2^{\mathbb{Z}^n})$. We also consider a similar strong markers theorem for more general generating sets. In dimension 2, this gives another proof of the fact that for any generating set $S\subseteq \mathbb{Z}^2$, there is a continuous proper edge $(2|S|+1)$-coloring of the Schreier graph of $F(2^{\mathbb{Z}^n})$ with generating set $S$.
\end{abstract}

\maketitle
\section{Introduction}    
The construction of marker sets is an indispensable tool in the study of descriptive set theory of countable Borel equivalence relations. What we now informally call {\em marker sets} used to be called {\em cross sections} (e.g. in \cite{We84}) or {\em complete sections} (e.g. in \cite{DJK94}). The informal terminology seems to have appeared first in \cite{JKL02} to emphasize the geometric nature of the orbits, but the concept had been studied much earlier (see, e.g., \cite{SS88} and \cite{BK96}).  

Recall that an equivalence relation $E$ on a Polish space $X$ is {\em Borel} if $E$ is a Borel subset of $X^2$, and $E$ is {\em countable} if every equivalence class of $E$ is countable. By considering the Bernoulli shift actions of countable discrete groups, we give structures to all countable Borel equivalence relations. In fact, by a theorem of Feldman--Moore 
(\cite[Theorem 1]{FM77}), any countable Borel equivalence relation is the orbit equivalence relation of a Borel action of some countable group. 

More concretely, let $\Gamma$ be a countable group and consider the {\em Bernoulli shift action} of $\Gamma$ with alphabet $2^{\mathbb{N}}$ defined as follows. This is the action of $\Gamma$ on $(2^{\mathbb{N}})^\Gamma$, where for any $g, h\in \Gamma$ and $x\in (2^\mathbb{N})^\Gamma$, 
$$ (g\cdot x)(h)=x(hg). $$
The space $(2^{\mathbb{N}})^\Gamma$ is a zero-dimensional Polish space, and the Bernoulli shift action is continuous.  By a theorem of Jackson--Kechris--Louveau (\cite[Proposition 4.2]{JKL02}), all Borel actions of $\Gamma$ can be equivariantly embedded into the Bernoulli shift action of $\Gamma$ with alphabet $2^\mathbb{N}$. Thus, in a sense, this single continuous action of $\Gamma$ contains the full information of all Borel actions of $\Gamma$ on Polish spaces.

In order to obtain geometric structures for the orbits, we often make two additional assumptions. The first is freeness of the action of $\Gamma$. Recall that an action of $\Gamma$ on $X$ is {\em free} if for any $x\in X$ and non-identity $g\in G$, $g\cdot x\neq x$. In the case of a free action of $\Gamma$, each orbit will naturally inherit the structure of $\Gamma$. Secondly, we often assume that $\Gamma$ is finitely generated. If we fix a finite generating set $S$ for $\Gamma$ and consider the Cayley graph of $(\Gamma, S)$, this graph structure will be tranferred to each orbit (in the case of a free action), which is called the {\em Schreier graph} of the action of $\Gamma$ with generating set $S$. In addition, the path metric on the Cayley graph (and therefore on the Schreier graph) thus gives each orbit a geometric structure. It is with this geometric structure that a marker set, or sometimes a sequence of marker sets, becomes the starting point of a more sophisticated construction or analysis.

In summary, the typical setup of our study is to consider a continuous and free action of a finitely generated group $\Gamma$ on a zero-dimensional Polish space $X$. One of the simplest such space $X$ is $F(2^\Gamma)$, the free part of the Bernoulli shift action of $\Gamma$ with alphbet $\{0,1\}$.

Suppose $\Gamma$ is a countable group acting continuously and freely on a Polish space $X$. A {\em marker set} or a {\em complete section} is a subset $Y\subseteq X$ such that $Y$ meets each orbit in $X$. Since the space $X$ is often zero-dimensional, we consider marker sets that are clopen, open, etc. or Borel in general. To avoid triviality, we often require a  marker set to be {\em co-complete}, i.e., its complement is also a complete section. 

For the marker sets to be useful, we often impose other regularity conditions on them. The following is a successful example (definitions of notation and terminology can be found in Section~\ref{sec:2}).

\begin{lemma}[Basic clopen marker lemma {\cite[Lemma 2.1]{GJ15}}]\label{lem:basic2} For any positive integer $d\geq 1$, there is a clopen set $M\subseteq F(2^{\mathbb{Z}^n})$ such that
\begin{enumerate}
\item[(1)] for any distinct $x, y\in M$ in the same orbit, $\rho(x,y)\geq d$;
\item[(2)] for any $x\in F(2^{\mathbb{Z}^n})$ there is $y\in M$ such that $\rho(x,y)<d$.
\end{enumerate}
\end{lemma}

However, there is a fine line dividing what is possible and what is not possible. In \cite{GJKS24} marker sets with very strong regularity properties were constructed, but in \cite{GJKS22} and \cite{GJKS23} marker sets with various regularity properties were proven not to exist. For example, it was proven in \cite[Theorem 1.2]{GJKS22} that for any countably infinite group $\Gamma$ there is no Borel marker set $M\subseteq F(2^\Gamma)$ such that for any finite $F\subseteq \Gamma$ and any $x\in F(2^\Gamma)$, there is some $y\in [x]$ with $F\cdot y\cap M=\varnothing$. In \cite[Theorem 1.6.4]{GJKS23} it was proven that there is no clopen single line section in $F(2^{\mathbb{Z}^2})$ (definitions can be found in Section~\ref{sec:5}). We will recall more of such results when we discuss applications of our main theorems in Section~\ref{sec:5}.

Our first main theorem in this paper is to demonstrate the existence of marker sets in $F(2^{\mathbb{Z}^n})$ with the following strong regularity property.

\begin{theorem} Let $n, d\geq 1$ be positive integers. Then there is a positive integer $D\geq d$ and a clopen subset $M\subseteq F(2^{\mathbb{Z}^n})$ such that
\begin{enumerate}
\item[(1)] for any distinct $x, y\in M$ in the same orbit, $\rho(x,y)\geq d$;
\item[(2)] for any $1\leq i\leq n$ and any $x\in F(2^{\mathbb{Z}^n})$, there are non-negative integers $a, b\leq D$ such that $ae_i\cdot x\in M$ and $-be_i\cdot x\in M$.
\end{enumerate}
\end{theorem}

We also show two applications of this theorem. One application is to construct a continuous proper edge $(2n+1)$-coloring of the Schreier graph on $F(2^{\mathbb{Z}^n})$ with the standard generating set. This was shown to exist in \cite{GWW25} by a different method. With our main theorem here, the construction of the edge coloring becomes easy. Another application is to construct a clopen tree section of $F(2^{\mathbb{Z}^n})$ for $n\geq 2$ that is both complete and co-complete. 

Our second main theorem in this paper is a strengthening of the first main theorem with a more general generating set.

\begin{theorem} Let $n, d\geq 1$ be positive integers and let $S\subseteq \mathbb{Z}^n$ be a finite generating set. Suppose for each $g\in S$, $\{1\leq i\leq n\colon \pi_i(g)\neq 0\}$ has size either $1$ or $n$. Then there is a positive integer $\Delta\geq d$ and a clopen subset $M\subseteq F(2^{\mathbb{Z}^n})$ such that
\begin{enumerate}
\item[(1)] for any distinct $x,y\in M$, $\rho(x,y)\geq d$;
\item[(2)] for any $g\in S$ and any $x\in F(2^{\mathbb{Z}^n})$, there are non-negative integers $a, b\leq \Delta$ such that $ag\cdot x\in M$ and $-bg\cdot x\in M$.
\end{enumerate}
\end{theorem}

An easy consequence is the existence of a continuous proper edge $(2|S|+1)$-coloring of the Schreier graph on $F(2^{\mathbb{Z}^n})$ with the generating set $S$. In the case of dimension $2$, this covers the case of arbitrary generating sets. Both these corollaries were known from \cite{GWW25}, but the proofs here are very different.

The rest of the paper is organized as follows. In Section \ref{sec:2} we give the basic definitions and introduce the notation to be used in the rest of the paper. In Section~\ref{sec:3} we prove a version of the main theorem in $\mathbb{Z}^n$. The proof, including the lemmas used in the proof, will be used in the proof of the first main theorem in Section~\ref{sec:4}. In Section~\ref{sec:4} we prove the first main theorem. In Section~\ref{sec:5} we present the two applications of the first main theorem. In Section~\ref{sec:6} we prove the second main theorem and its corollaries.

\section{Preliminaries\label{sec:2}}

\subsection{Basic concepts in $\mathbb{Z}^n$}\ 

Let $n\geq 1$ be an integer. For $x=(x_1,\dots, x_n)\in\mathbb{Z}^n$ and $1\leq i\leq n$, let $x(i)=x_i$. For $1\leq i\leq n$, we also let $\pi_i\colon \mathbb{Z}^n\to \mathbb{Z}$ be the projection map $\pi_i(x)=x(i)$ and let $\sigma_i\colon \mathbb{Z}^n\to\mathbb{Z}^{n-1}$ be the projection map 
$$ \sigma_i(x)=(x(1),\dots, x(i-1), x(i+1), \dots, x(n)). $$ 
For $1\leq i\leq n$, let $e_i\in\mathbb{Z}^n$ be defined by
$$ e_i(j)=\left\{\begin{array}{ll} 1, & \mbox{ if $i=j$,} \\ 0, & \mbox{ otherwise.}\end{array}\right. $$
We view $e_i$ as a vector and refer to the direction $e_i$ in the usual sense.

For $x, y\in \mathbb{Z}^n$, we define
$$ \rho(x,y)=\sup\left\{ |x(i)-y(i)|\colon 1\leq i\leq n\right\} $$
and refer to $\rho(x,y)$ as the {\em distance} between $x$ and $y$. Let $\overline{0}=(0,\dots, 0)\in\mathbb{Z}^n$. For $x\in\mathbb{Z}^n$, let
$$ \|x\|=\rho(x,\overline{0}). $$
 For $A, B\subseteq \mathbb{Z}^n$, we define
$$ \rho(A, B)=\inf\{\rho(x,y)\colon x\in A, y\in B\} $$
and refer to $\rho(A, B)$ as the {\em distance} between $A$ and $B$. One can similarly define $\rho(x, A)$ for $x\in \mathbb{Z}^n$ and $A\subseteq \mathbb{Z}^n$. Note that $\rho$ denotes the distance function for all dimensions.

For integers $a<b$, we use $[a,b]$ to denote the set $\{t\in\mathbb{Z}\colon a\leq t\leq b\}$ and still call it an {\em interval}. An {\em $n$-dimensional rectangle} $R$ is a subset of $\mathbb{Z}^n$ of the form
$$ R=[a_1,b_1]\times \cdots\times[a_n, b_n], $$
where $a_1<b_1$, $\dots$, $a_n<b_n$ are integers. For an $n$-dimensional rectangle $R$ as above, its {\em side length} in the direction $e_i$ is defined as $b_i-a_i$. 

Note that an $n$-dimensional rectangle $R$ has $2^n$ many {\em extreme points} (or more intuitively, {\em corners}). The following is a fixed enumeration of them. For $1\leq k\leq 2^n$, let
$$ x_k(i)=\left\{\begin{array}{ll} a_i, & \mbox{ if the $i$-th least-significant digit of  } \\
& \mbox{ the binary expansion of $k-1$ is 0,} \\
b_i, & \mbox{ otherwise,}\end{array}\right. $$
for all $1\leq i\leq n$. Then $\{x_k\colon 1\leq k\leq 2^n\}$ is the set of all extreme points of $R$. We call $x_1,\dots, x_{2^n}$ the {\em canonical} enumeration of the extreme points of $R$.

We will also work with {\em generalized $n$-dimensional rectangles}, which are subsets of $\mathbb{Z}^n$ of the form
$$ G=[a_1,b_1]\times \cdots \times [a_n,b_n] $$
where $a_1\leq b_1$, $\dots$, $a_n\leq b_n$ are integers. In other words, side lengths of a generalized $n$-dimensional rectangle could be $0$, in which case the generalized $n$-dimensional rectangle is indeed a degenerate rectangle in $\mathbb{Z}^n$ whose actual dimension could be lower. A {\em generalized interval} is an interval or a singleton.

\subsection{Basic concepts in $F(2^{\mathbb{Z}^n})$}\

The main theorems of the paper will be about the ({\em Bernoulli}) {\em  shift action} of $\mathbb{Z}^n$ on $2^{\mathbb{Z}^n}$ defined as follows. For $g, h\in \mathbb{Z}^n$ and $x\in 2^{\mathbb{Z}^n}$, define
$$ (g\cdot x)(h)=x(g+h). $$
The {\em free part} of the action is 
$$ F(2^{\mathbb{Z}^n})=\left\{x\in 2^{\mathbb{Z}^n}\colon \forall g\in\mathbb{Z}^n\ (g\neq \overline{0}\rightarrow g\cdot x\neq x)\right\}. $$
With the product topology, $2^{\mathbb{Z}^n}$ is a Polish space. $F(2^{\mathbb{Z}^n})$ is a $G_\delta$ subspace of $2^{\mathbb{Z}^n}$, hence it is also a Polish space. The shift action of $\mathbb{Z}^n$ on $F(2^{\mathbb{Z}^n})$ is continuous. 

For any $x\in F(2^{\mathbb{Z}^n})$, the {\em orbit} of $x$ is $[x]=\mathbb{Z}^n\cdot x=\{g\cdot x\colon g\in \mathbb{Z}^n\}$. Because of freeness, every orbit in $F(2^{\mathbb{Z}^n})$ can be viewed as a copy of $\mathbb{Z}^n$, and therefore all the concepts in $\mathbb{Z}^n$ apply to orbits of $F(2^{\mathbb{Z}^n})$. For example, we can speak of $\rho(x,y)$ for $y\in[x]$ as the {\em distance} between $x$ and $y$ in the same orbit. Formally, $\rho(x,y)=\|g\|$, where $g\in\mathbb{Z}^n$ is the unique element of the group such that $g\cdot x=y$. One can also talk about {\em intervals}, {\em generalized intervals}, {\em $n$-dimensional rectangles}, {\em generalized $n$-dimensional rectangles}, etc. in an orbit. Formally, an {\em interval} in the direction $e_i$ in $F(2^{\mathbb{Z}^n})$ is a set of the form $I\cdot x$ for some $x\in F(2^{\mathbb{Z}^n})$ and some interval $I$ in the direction $e_i$ in $\mathbb{Z}^n$. The other concepts are similarly defined.

We will also consider the {\em Schreier graph} $G$ on $F(2^{\mathbb{Z}^n})$, which is the graph defined as
$$ \{x,y\}\in E(G)\iff \exists 1\leq i\leq n\ \exists \lambda\in\{-1,1\}\ (\lambda e_i\cdot x=y). $$
We say an edge $\{x,y\}\in E(G)$ is {\em parallel} to $e_i$ if either $e_i\cdot x=y$ or $-e_i\cdot x=y$. 
Here $E(G)$ can be identified as the image of a subspace of $F(2^{\mathbb{Z}^n})^2$ under a $2$-to-$1$ map $\theta$, and thus can be given the smallest topology to make $\theta$ continuous. 
For positive integer $k\geq 2$, a {\em proper edge $k$-coloring} of $G$ is a map $c\colon E(G)\to \{1,\dots, k\}$ such that for any distinct $x, y, z\in F(2^{\mathbb{Z}^n})$, if $\{x,y\}, \{y, z\}\in E(G)$, then $c(\{x,y\})\neq c(\{y, z\})$. With the abovementioned topology on $E(G)$, we can refer to {\em continuous} proper edge $k$-colorings. The {\em continuous edge chromatic number} of $F(2^{\mathbb{Z}^n})$ is the smallest integer $k$ such that there exists a continuous proper edge $k$-coloring of $F(2^{\mathbb{Z}^n})$. It was shown in \cite{GWW25} that the continuous edge chromatic number of $F(2^{\mathbb{Z}^n})$ is exactly $2n+1$; in particular, there is a continuous proper edge $(2n+1)$-coloring of $F(2^{\mathbb{Z}^n})$.

The above definitions and notation do not exhaust all that we use in the rest of the paper; more definitions will be made and notation defined as they become necessary later.

\section{Strong marker sets in $\mathbb{Z}^n$\label{sec:3}}

The following lemma is the most basic step of our construction of strong markers.

\begin{lemma}\label{lem:1direction}
Let $n, d\geq 1$ be positive integers and let $1\leq i\leq n$. For any generalized $n$-dimensional rectangle $R$ with side length at least $2d^n-d$ in the direction $e_i$, there is a subset $M\subseteq R$ such that
\begin{enumerate}
\item[(1)] for any distinct $x,y\in M$, $\rho(x,y)\geq d$;
\item[(2)] for any $x\in R$, there is $a\in\mathbb{Z}$ such that $x+ae_i\in M$.
\end{enumerate}
\end{lemma}
\begin{proof} By induction on $n$. For $n=1$ the lemma trivially holds if we take $M\subseteq R$ to be a singleton. Now assume $n>1$. Without loss of generality, we may assume that $i=1$ and 
$$R=[0,2d^n-d-1]\times[0,b_2]\times\cdots\times[0,b_n]$$ where $b_i\geq 0$ for all $2\leq i\leq n$. 
Consider 
$$ R_{n-1}=[0, 2d^{n-1}-d-1]\times [0, b_2]\times \cdots \times [0, b_{n-1}]. $$
By the inductive hypothesis, there is a subset $M_{n-1}\subseteq R_{n-1}$ such that (1) and (2) hold for $M_{n-1}$. Now define
$$ M=\bigcup_{t\in [0, b_n]} (M_{n-1}+2(t\!\!\!\!\mod d)d^{n-1}e_1)\times\{t\}. $$
Since $0\leq (t\!\!\mod d)\leq d-1$, we have that for any $x\in M$, 
$$0\leq x(1)\leq 2d^{n-1}-d-1+2(d-1)d^{n-1}=2d^n-d-1;$$ hence $M\subseteq R$.
Now it is easy to see that (2) holds for $M$. For (1), just note that for any distinct $x, y\in M$, if $x(n)=y(n)$ then $\rho(x,y)\geq d$ by (1) for $M_{n-1}$; if $x(n)\neq y(n)$, then $|x(1)-y(1)|\geq d$, hence $\rho(x,y)\geq d$.
\end{proof}

We will informally call the set $M$ a {\em marker set} and its elements {\em marker points}.  Further constructions of marker sets will involve identifying a certain collection of generalized $n$-dimensional subrectangles of a given $n$-dimensional rectangle (we call the technique ``packaging") and applying Lemma~\ref{lem:1direction} to each subrectangle in the collection. To guarantee (1) we need to appropriately ``space out" the generalized $n$-dimensional subrectangles in the collection (we call this technique ``spacing"). The following two lemmas illustrate the ``packaging and spacing" technique. 

\begin{lemma} \label{lem:dim1packagingandspacing}Let $m\geq 0$ be an integer and let $d, k\geq 1$ be positive integers. Let $I$ be an interval of $\mathbb{Z}$. Let $\mathcal{J}$ be a collection of $m$ many subintervals of $I$ each of whose length is at most $d$. Suppose the length of $I$ is at least $3d(2m+k+1)$. Then there is a collection $\mathcal{K}$ of $k$ many pairwise disjoint subintervals of $I$ such that 
\begin{enumerate}
\item[(i)] for each $K\in \mathcal{K}$, the length of $K$ is at least $d$;
\item[(ii)] for any distinct $K_1, K_2\in\mathcal{K}$, $\rho(K_1, K_2)\geq d$;
\item[(iii)] for any $J\in \mathcal{J}$ and $K\in\mathcal{K}$, $\rho(J, K)\geq d$.
\end{enumerate}
\end{lemma}
\begin{proof} Divide $I$ up into consecutive disjoint intervals of length $3d$, with at most one interval at the end whose length is at most $3d$. So there are at least $2m+k+1$ many such intervals. Now $\bigcup \mathcal{J}$ can intersect at most $2m$ many such intervals. Thus there are at least $k$ many such intervals with no intersections with $\bigcup\mathcal{J}$. Let $\mathcal{K}$ be the middle $\frac{1}{3}$ of these $k$ many intervals of length $3d$. Then it is easy to verify that (i)--(iii) hold for $\mathcal{K}$.
\end{proof}

\begin{lemma} \label{lem:1directionmultiple}
Let $m\geq 0$ be an integer, let $n, d, k\geq 1$ be positive integers and let $1\leq i\leq n$. Let $R$ be an $n$-dimensional rectangle. Let $\mathcal{P}$ be a set of  $k$ many pairwise disjoint generalized $(n-1)$-dimensional rectangles such that  for every $P\in \mathcal{P}$, $P\subseteq \sigma_i(R)$. Let $\mathcal{J}$ be a collection of $m$ many subintervals of $\pi_i(R)$ each of whose length is at most $d$. Suppose $R$ has side length at least $3d(2m+k+1)$ in the direction $e_i$. Then there is an assignment $P\mapsto K_P$ from $\mathcal{P}$ to subintervals of $\pi_i(R)$ such that
\begin{enumerate}
\item[(1)] for any $P\in\mathcal{P}$, the length of $K_P$ is at least $d$;
\item[(2)] for any distinct $P, Q\in\mathcal{P}$, $\rho(K_P, K_Q)\geq d$;
\item[(3)] for any $J\in \mathcal{J}$ and $P\in\mathcal{P}$, $\rho(J, K_P)\geq d$.
\end{enumerate}
In particular, there is a collection $\mathcal{Q}$ of $k$ many generalized $n$-dimensional subrectangles of $R$ such that
\begin{enumerate}
\item[(4)] $\{\sigma_i(Q)\colon Q\in\mathcal{Q}\}=\mathcal{P}$;
\item[(5)] for any $Q\in\mathcal{Q}$, $\pi_i(Q)$ has length at least $d$;
\item[(6)] for any distinct $Q, S\in\mathcal{Q}$, $\rho(Q,S)\geq d$;
\item[(7)] for any $Q\in\mathcal{Q}$ and subrectangle $S$ of $R$ with $\pi_i(S)\subseteq \bigcup\mathcal{J}$, $\rho(Q, S)\geq d$.
\end{enumerate}
\end{lemma}

\begin{proof} Without loss of generality assume $i=n$. By Lemma~\ref{lem:dim1packagingandspacing}, we obtain $k$ many pairwise disjoint subintervals of $I$ with properties (i)--(iii). Let $P\mapsto K_P$ be an arbitrary bijection from $\mathcal{P}$ to $\mathcal{K}$. Then (1)--(3) hold for $\mathcal{K}$. Let $\mathcal{Q}=\{P\times K_P\colon P\in\mathcal{P}\}$. Then (4)--(7) hold.
\end{proof}

In the above lemma elements of $\mathcal{Q}$ are ``packages." Applying Lemma~\ref{lem:1direction} to each of the packages, we obtain a marker set with the desired properties. This construction will be repeatedly used in the rest of the paper. In fact, the circumstance under which Lemma~\ref{lem:1directionmultiple} will be applied is when $\sigma_i(R)\setminus \bigcup\mathcal{P}$ is the union of $L$ many generalized $(n-1)$-dimensional rectangles (which are not necessarily pairwise disjoint). The next lemma gives an estimate for $|\mathcal{P}|$ given $L$.

\begin{lemma}\label{lem:division} Let $n, L\geq 1$ be positive integers. Suppose $R$ is an $n$-dimensional rectangle and $S\subseteq R$ is the union of $L$ many $n$-dimensional rectangles. Then $R\setminus S$ can be written as the union of at most $(2L+1)^n$ many pairwise disjoint generalized $n$-dimensional rectangles.
\end{lemma}

\begin{proof} Write $S=\bigcup_{1\leq \ell\leq L} S_{\ell}$ as the union of $L$ many generalized $n$-dimensional rectangles. For each $1\leq i\leq n$ and $1\leq \ell\leq L$, $\pi_i(S_{\ell})$ divides the interval $\pi_i(R)$ into at most three pairwise disjoint generalized intervals. Taking all $\pi_i(S_{\ell})$, $1\leq \ell\leq L$, into account together, the interval $\pi_i(R)$ is divided into at most $2L+1$ many pairwise disjoint generalized intervals. Thus we can write $R$ as the union of at most $(2L+1)^n$ many pairwise disjoint generalized $n$-dimensional rectangles. Let $\mathcal{P}$ be the set of these generalized $n$-dimensional rectangles. Then $|\mathcal{P}|\leq (2L+1)^n$. Note that each $S_{\ell}$ is the union of some elements of $\mathcal{P}$, from which it follows that $R\setminus S$ is the union of some elements of $\mathcal{P}$.
\end{proof}

Now we are ready to prove the following theorem, which is a desirable strengthening of Lemma~\ref{lem:1direction}. Although the conclusion of the following theorem is neither necessary nor sufficient for our main theorems, its proof will be useful in the proof of our main theorems.

\begin{theorem}\label{thm:strongZn} Let $n, d_0\geq 1$ be positive integers. Then there is a positive integer $D_0$ such that for any $n$-dimensional rectangle $R$ with side length at least $D_0$ in each direction, there is a subset $M\subseteq R$ such that
\begin{enumerate}
\item[(1)] for any distinct $x,y\in M$, $\rho(x,y)\geq d_0$;
\item[(2)] for any $1\leq i\leq n$ and any $x\in R$, there is $a_i\in \mathbb{Z}$ such that $x+a_ie_i\in M$.
\end{enumerate}
\end{theorem}

\begin{proof} We first give an informal description of the construction of $M\subseteq R$. This helps in determining how large $D_0$ needs to be. In the construction of $M$, we inductively define a sequence of collections $\mathcal{P}_i$, $1\leq i\leq n$, of pairwise disjoint generalized $n$-dimensional subrectangles of $R$ and a decreasing sequence of positive integers $d_i>2d_0^n-d_0$ so that 
\begin{enumerate}
\item[(i)] $\bigcup\{\sigma_i(P)\colon P\in\mathcal{P}_i\}=\sigma_i(R)$;
\item[(ii)] for any $P\in \mathcal{P}_i$, the length of $\pi_i(P)$ is at least $2d_0^n-d_0$ but at most $d_i$;
\item[(iii)] for distinct $P, Q\in\mathcal{P}_i$, $\rho(P, Q)\geq d_i$;
\item[(iv)] for distinct $1\leq j<i\leq n$, if $P\in\mathcal{P}_i$ and $Q\in \mathcal{P}_j$, then $\rho(P, Q)\geq d_i$.
\end{enumerate}
Granting such sequences, we then apply Lemma~\ref{lem:1direction} to each $P\in \mathcal{P}_i$ in the direction $e_i$ to obtain a marker set $M_P$. Then the union
$$ M=\bigcup\{M_P\colon P\in\mathcal{P}_i, 1\leq i\leq n\} $$
is as required. 

We fix a sequence of positive integers $d_1,\dots, d_n$ so that
$$ 2d_0^n-d_0<d_n<d_{n-1}<\cdots<d_2<d_1 $$
and for all $1\leq i< n$, 
$$ d_i>5(n-i)d_{i+1}>5(d_{i+1}+\cdots+d_n). $$
For example, we can let $d_n=6nd_0^n$ and $d_i=6nd_{i+1}$ for all $1\leq i< n$. 
Then $$ d_1=(6nd_0)^n. $$

In the construction of $\mathcal{P}_{i+1}$, consider $\mathcal{Q}_i=\mathcal{P}_1\cup\cdots\cup \mathcal{P}_i$. We use Lemma~\ref{lem:division} to divide $\sigma_i(R)$ into at most $(2|\mathcal{Q}_i|+1)^{n-1}$ many (note: this number is subject to revisions later) pairwise disjoint generalized $(n-1)$-dimensional rectangles, and then use (a variation of) the packaging and spacing technique of Lemma~\ref{lem:1directionmultiple} to obtain the new collection $\mathcal{P}_{i+1}$. For this plan to work, we set inductively 
$$ N_1=4^{n-1}, \ N_{i+1}=4^{n-1}(2N_i+1)^{n-1}+2N_i+1. $$
Now let
$$D_0=4N_nd_1.$$  

For our plan to work, we will maintain the following inductive hypotheses on $\mathcal{P}_i$ in addition to (i)--(iv) above:
\begin{enumerate}
\item[(v)] for any $P\in\mathcal{P}_i$ and any $1\leq j\leq n$, the side lengths of $P$ in the direction $e_j$ is at most half of the side length of $R$ in the direction $e_j$;
\item[(vi)] $|\mathcal{Q}_i|\leq N_i$.
\end{enumerate}

Now suppose $R$ is an $n$-dimensional rectangle whose side lengths in all directions are at least $D_0$. Let $w_i$ be the side length of $R$ in the direction $e_i$. We are now ready to define the collections $\mathcal{P}_i$ inductively. 

To define $\mathcal{P}_1$ we consider $\sigma_1(R)$. Write $\sigma_1(R)$ as the union of $4^{n-1}$ many pairwise disjoint $(n-1)$-dimensional rectangles whose side lengths in the direction $e_i$, $1<i\leq n$, are approximately $\frac{1}{4}w_i$. Since $w_i\geq D_0=4N_nd_1$, these side lengths are at least $N_nd_1$, but certainly no more than $\frac{1}{2}w_i$. Applying Lemma~\ref{lem:1directionmultiple} with $m=0$, $d=d_1$ and $k=4^{n-1}$, we obtain a collection $\mathcal{P}_1$ of $n$-dimensional subrectangles of $R$. To ensure (ii), we might need to slightly shrink the subrectangles in the direction $e_1$ so that their side lengths in this direction are at least $2d_0^n-d_0$ and at most $d_1$. It is now easy to see that $\mathcal{P}_1$ satisfies all inductive hypotheses (i)--(v). For (vi), note that $|\mathcal{Q}_1|=|\mathcal{P}_1|=4^{n-1}=N_1$. 

For the general inductive step $1\leq i<n$, assume we have defined $\mathcal{P}_1,\dots, \mathcal{P}_{i}$ so that the inductive hypotheses (i)--(vi) hold. Let $\mathcal{Q}_{i}=\mathcal{P}_1\cup\cdots\cup\mathcal{P}_i$. By the inductive hypothesis (vi), $|\mathcal{Q}_i|\leq N_i$. Then by the inductive hypotheses (iii) and (iv), for any distinct $P, Q\in \mathcal{Q}_i$, $\rho(P, Q)\geq d_i$. Now consider $\sigma_{i+1}(R)$. We need to define some special packages. For each $P\in \mathcal{Q}_i$, define
$$ T_P=\left\{ x\in R\colon \rho(x, P)\leq 2d_{i+1}\right\}. $$
For distinct $P, Q\in \mathcal{Q}_i$, we have $\rho(T_P, T_Q)\geq d_i-4d_{i+1}>d_{i+1}$. A moment of reflection shows that for each $P\in\mathcal{Q}_i$, we can find an $n$-dimensional rectangle $R_P$ so that
\begin{enumerate}
\item[(a)] $R_P\subseteq T_P$ and $\sigma_{i+1}(R_P)=\sigma_{i+1}(T_P)$;
\item[(b)] the side length of $R_P$ in the direction of $e_{i+1}$ is at least $2d_0^n-d_0$ and at most $d_{i+1}$;
\item[(c)] $\rho(R_P, P)\geq d_{i+1}$.
\end{enumerate}
Note that $|\{R_P\colon P\in\mathcal{Q}_i\}|=|\mathcal{Q}_i|\leq N_i$. Let $\mathcal{J}=\{\pi_{i+1}(R_P)\colon P\in\mathcal{Q}_i\}$. Then $\mathcal{J}$ is a collection of at most $N_i$ many subintervals of $\pi_{i+1}(R)$ each of whose length is at most $d_{i+1}$.

Next let $S=\bigcup\{\sigma_{i+1}(R_P)\colon P\in\mathcal{Q}_i\}$. Applying Lemma~\ref{lem:division} to $\sigma_{i+1}(R)\setminus S$, we obtain at most $(2|\mathcal{Q}_i|+1)^{n-1}$ many pairwise disjoint generalized $(n-1)$-dimensional rectangles. To further satisfy (v), we subdivide each of these generalized $(n-1)$-dimensional rectangles into up to $4^{n-1}$ many subsubrectangles so that their side lengths are at most $\frac{1}{4}w_i+1\leq \frac{1}{2}w_i$. Let $\mathcal{S}$ denote the resulting collection of generalized $(n-1)$-dimensional rectangles. Then 
$$ |\mathcal{S}|\leq 4^{n-1}(2|\mathcal{Q}_i|+1)^{n-1}. $$
Applying Lemma~\ref{lem:1directionmultiple} with $m=|\mathcal{J}|\leq N_i$, $d=d_{i+1}$ and $k=|\mathcal{S}|\leq 4^{n-1}(2N_i+1)^{n-1}$, we obtain a collection $\mathcal{R}$ of generalized $n$-dimensional subrectangles of $R$ so that 
\begin{enumerate}
\item[(d)] for any $P\in \mathcal{R}$, the side length of $P$ in the direction $e_{i+1}$ is at least $2d_0^n-d_0$ and at most $d_{i+1}$;
\item[(e)] for any distinct $P, Q\in \mathcal{R}$, $\rho(P,Q)\geq d_{i+1}$;
\item[(f)] for any $P\in\mathcal{R}$ and $Q\in\mathcal{Q}_i$, $\rho(P, Q)\geq d_{i+1}$;
\item[(g)] for any $P\in\mathcal{R}$ and $Q\in\mathcal{Q}_i$, $\rho(P, R_Q)\geq d_{i+1}$. 
\end{enumerate}
For this application of Lemma~\ref{lem:1directionmultiple} to be valid, we note that
$$\begin{array}{rcl} w_{i+1}\geq D=4N_nd_1\!\!\!\!&\geq &\!\!\!\!3N_{i+1}d_{i+1} \\
&=&\!\!\!\! 3d_{i+1}(4^{n-1}(2N_i+1)^{n-1}+2N_i+1)\geq 3d(2m+k+1). 
\end{array}$$ 
Finally, define
$$ \mathcal{P}_{i+1}=\mathcal{R}\cup\{R_P\colon P\in\mathcal{Q}_i\}. $$
Then 
$$|\mathcal{Q}_{i+1}|\leq |\mathcal{R}|+2|\mathcal{Q}_i|=|\mathcal{S}|+2|\mathcal{Q}_i|\leq 4^{n-1}(2N_i+1)^{n-1}+2N_i\leq N_{i+1}. $$
Thus the inductive hypotheses are maintained.
\end{proof}

\section{Strong marker sets in $F(2^{\mathbb{Z}^n})$\label{sec:4}}

In this section we will prove the following main theorem of the paper. 

\begin{theorem} \label{thm:main} Let $n, d_0\geq 1$ be positive integers. Then there is a positive integer $D$ and a clopen subset $M\subseteq F(2^{\mathbb{Z}^n})$ such that
\begin{enumerate}
\item[(1)] for any distinct $x, y\in M$ in the same orbit, $\rho(x,y)\geq d_0$;
\item[(2)] for any $1\leq i\leq n$ and any $x\in F(2^{\mathbb{Z}^n})$, there are non-negative integers $a_i, b_i\leq D$ such that $a_ie_i\cdot x\in M$ and $-b_ie_i\cdot x\in M$.
\end{enumerate}
\end{theorem}

The rest of this section is devoted to a proof of Theorem~\ref{thm:main}. We first recall the following basic fact.

\begin{lemma}[Marker regions lemma {\cite[Theorem 3.1]{GJ15}}] \label{lem:basicmarker}
Let $d\geq 1$ be a positive integer. Then there is a clopen equivalence relation $E^n_d$ on $F(2^{\mathbb{Z}^n})$ such that each of the $E^n_d$-equivalence classes is an $n$-dimensional rectangle with side lengths either $d$ or $d+1$.
\end{lemma}

We informally call the $E^n_d$-equivalence classes {\em marker regions}. The above lemma can be informally stated as there are clopen marker regions which are $n$-dimensional rectangles with side lengths either $d$ or $d+1$. 

We will also use the basic clopen marker lemma from \cite{GJ15} which we recalled in the Introduction. For the convenience of the reader we restate it here.

\begin{lemma}[Basic clopen marker lemma {\cite[Lemma 2.1]{GJ15}}]\label{lem:basic2} For any positive integer $d\geq 1$, there is a clopen set $M\subseteq F(2^{\mathbb{Z}^n})$ such that
\begin{enumerate}
\item[(1)] for any distinct $x, y\in M$ in the same orbit, $\rho(x,y)\geq d$;
\item[(2)] for any $x\in F(2^{\mathbb{Z}^n})$ there is $y\in M$ such that $\rho(x,y)<d$.
\end{enumerate}
\end{lemma}

To prove Theorem~\ref{thm:main}, we fix positive integers $n, d_0\geq 1$. Let $d_1, \dots, d_n$ be the same sequence of positive integers as defined in the proof of Theorem~\ref{thm:strongZn}. Redefine the sequence $N_i$ inductively by
$$ N_1=4^{n-1}, \ N_{i+1}=4^{n-1}(2N_i+1)^{n-1}+(n-1)2^{n+1}N_i+5. $$
Let $$D_1=4N_nd_1$$ and $$D=2D_1.$$

Applying Lemma~\ref{lem:basicmarker} to $d=D_1$, we obtain a clopen equivalence relation $E$ on $F(2^{\mathbb{Z}^n})$ whose marker regions are $n$-dimensional rectangles with side lengths either $D_1$ or $D_1+1$. For each marker region $R$, let $x_R$ be the ``least" corner of $R$, i.e., the extreme point of $R$ with index $1$ in the canonical enumeration of all extreme points of $R$. Let $X=\{x_R\colon \mbox{ $R$ is a marker region}\}$. Then $X$ is a clopen subset of $F(2^{\mathbb{Z}^n})$. 

Let $\Delta>6D_1$ be an integer. Applying Lemma~\ref{lem:basic2} to $d=\Delta$, we obtain a marker set $M_\Delta$. Enumerate the set $\{g\in\mathbb{Z}^n\colon \|g\|\leq \Delta\}$ as $g_1,\dots, g_H$. We define a partition $\{X_1,\dots, X_H\}$ of $X$ as follows:
$$\begin{array}{rcl} X_1\!\!\!\!&=&\!\!\!\! g_1\cdot M_\Delta\cap X, \\
X_h\!\!\!\!&=&\!\!\!\! g_h\cdot M_\Delta\cap X \setminus \bigcup_{1\leq j<h} X_j, \mbox{ for $1<h\leq H$.}
\end{array}$$
Note that for any $1\leq h\leq H$ and any distinct marker regions $R$ and $T$ with $x_R, x_T\in X_h$, $\rho(R, T)> 2D_1$ since $\rho(x_R, x_T)>6D_1$; in particular, $R$ and $T$ are not adjacent to each other. 

Our construction of $M$ will take place in $n$ many rounds, and each round consists of $H$ many steps. For round $i$, $1\leq i\leq n$, and step $h$, $1\leq h\leq H$, we consider those marker regions $R$ with $x_R\in X_h$ and construct a collection $\mathcal{P}_i^R$ of pairwise disjoint generalized $n$-dimensional subrectangles of $R$. As we noted above, for the same round and step, the distinct marker regions considered are not adjacent to each other, and hence the constructions of their respective $\mathcal{P}^R_i$s do not affect each other. 

The construction of $\mathcal{P}_i^R$ is similar to that in the proof of Theorem~\ref{thm:strongZn}. They will satisfy the following inductive hypotheses for any marker region $R$:
\begin{enumerate}
\item[(i)] $\bigcup\{\sigma_i(P)\colon P\in\mathcal{P}_i^R\}=\sigma_i(R)$;
\item[(ii)] for any $P\in\mathcal{P}_i^R$, the length of $\pi_i(P)$ is at least $2d_0^n-d_0$ and at most $d_i$;
\item[(iii)] for distinct $P, Q\in\mathcal{P}_i^R$, $\rho(P, Q)\geq d_i$;
\item[(iv)] for distinct $1\leq j<i\leq n$, if $P\in \mathcal{P}_i^R$ and $Q\in\mathcal{P}_j^R$, then $\rho(P, Q)\geq d_i$;
\item[(v)] for any $P\in\mathcal{P}_i^R$, all side lengths of $P$ are at most $\frac{1}{2}D_1$;
\item[(vi)] $|\mathcal{P}_i^R\cup\cdots\cup \mathcal{P}_1^R|\leq N_i$.
\item[(vii)] for any marker region $T\neq R$, for any $1\leq j\leq i\leq n$, for any $P\in \mathcal{P}_i^R$ and $Q\in \mathcal{P}_j^T$, $\rho(P, Q)\geq d_i$.
\end{enumerate}
Note that the first six of these inductive hypotheses are similar to those in the proof of Theorem~\ref{thm:strongZn}. The last one is new and will require some work to be maintained. Maintaining this inductive hypothesis was the main motivation to redefine the $N_i$ sequence.

We are now ready to define these sequences of packages. In the first round we define $\mathcal{P}_1^R$ for all marker regions $R$. This is done in $H$ many steps. In each step $h$, $1\leq h\leq H$, we simultaneously consider all those marker regions $R$ with $x_R\in X_h$. Recall that any two of them are not adjacent. So let $R$ be such a marker region.  Write $\sigma_1(R)$ as the union of $4^{n-1}$ many pairwise disjoint $(n-1)$-dimensional rectangles whose side lengths in all direction $e_i$, $1<i\leq n$, are either $\frac{1}{4}D=N_nd_1$ or $N_nd_1+1$. Thus in the application of Lemma~\ref{lem:1directionmultiple} we will have $k=4^{n-1}$. Next we compute $m$. For this we note that there are $2(n-1)$ many faces of $R$ whose normal vector is not $\pm e_1$, and $R$ can be adjacent to $2^{n-1}$ many other marker regions across each of these faces. Thus there are at most $2(n-1)2^{n-1}N_1$ many elements $P$ of $\mathcal{P}_1^T$, if it is already defined, where $T$ is a marker region adjacent to $R$, such that $\pi_1(P)\cap\pi_1(R)\neq\varnothing$. Let $I_1, I_2$ be subintervals of $\pi_1(R)$ whose lengths are $d_1$ such that $I_1, I_2$ are at the two ends of the interval $\pi_1(R)$. Let
$$\begin{array}{rl} \mathcal{J}_0= \{\pi_1(P)\cap \pi_1(R)\colon  & \!\!\!\!\mbox{ there is a marker region $T$ adjacent to $R$}\\
& \!\!\!\! \mbox{ whose $\mathcal{P}_1^T$ has been defined, and $P\in \mathcal{P}_1^T$}\}. 
\end{array}$$
and 
$$ \mathcal{J}=\mathcal{J}_0\cup\{I_1,I_2\}. $$
Then
$$ |\mathcal{J}|\leq 2(n-1)2^{n-1}N_1+2=(n-1)2^nN_1+2. $$
Now apply Lemma~\ref{lem:1directionmultiple} with $m=|\mathcal{J}|\leq (n-1)2^nN_1+2$, $d=d_1$, and $k=4^{n-1}$. This can be done since 
$$ D_1=4N_nd_1\geq 3N_2d_1\geq 3d_1(4^{n-1}+(n-1)2^{n+1}N_1+5)\geq 3d_1(2m+k+1). $$
We thus obtain a collection $\mathcal{P}_1^R$ with properties (i)--(vii).

For the general inductive round $1\leq i<n$, assume we have define $\mathcal{P}_1^R,\dots, \mathcal{P}_i^R$ for all marker regions $R$ so that the inductive hypotheses (i)--(vii) hold. Let $\mathcal{Q}^R_i=\mathcal{P}_1^R\cup\cdots\cup\mathcal{P}_i^R$. By the inductive hypotheses (vi), $\mathcal{Q}_i^R|\leq N_i$. We define the collection $\mathcal{R}^R$ of special packages the same way as in the proof of Theorem~\ref{thm:strongZn}. Then $$|\mathcal{R}^R|=|\mathcal{Q}_i^R|\leq N_i.$$ We also define the collection $\mathcal{S}^R$ of generalized $(n-1)$-dimensional rectangles in the same way as in the proof of Theorem~\ref{thm:strongZn}. Then $$|\mathcal{S}^R|\leq 4^{n-1}(2N_i+1)^{n-1}.$$ Finally, we define a collection $\mathcal{J}^R$ of intervals in the direction $e_i$ similar to the proof in the first round above. Then 
$$ |\mathcal{J}|\leq 2(n-1)2^{n-1}N_i+2. $$
Applying Lemma~\ref{lem:1directionmultiple} with $m=|\mathcal{J}|$, $d=d_{i+1}$, and $k=|\mathcal{S}^R|$, which is valid since
$$\begin{array}{rl} D_1=4N_nd_1\geq &\!\!\! 3N_{i+1}d_{i+1} \\
=&\!\!\! 3d_{i+1}(4^{n-1}(2N_i+1)^{n-1}+(n-1)2^{n+1}N_i+5)\geq 3d(2m+k+1), 
\end{array}$$
we obtain the collection $\mathcal{P}_{i+1}^R$ with the desired properties.

Finally, we apply Lemma~\ref{lem:1direction} to all packages identified by $\mathcal{P}^R_i$ to obtain a marker set $M$. This $M$ is clopen and has the desired properties. This finishes the proof of Theorem~\ref{thm:main}.

\section{Applications of strong marker sets\label{sec:5}}

In this section we present two applications of Theorem~\ref{thm:main}. The first application is to give a continuous proper edge $(2n+1)$-coloring of the Schreier graph on $F(2^{\mathbb{Z}^n})$. This was proved in \cite{GWW25} by a different method. It was shown in \cite{GWW25} that the continuous edge chromatic number of $F(2^{\mathbb{Z}^n})$ is exactly $2n+1$.

\begin{theorem}[\cite{GWW25}]\label{thm:coloring}
  For any integer $n\geq 1$, there is a continuous proper edge $(2n+1)$-coloring of $F(2^{\mathbb{Z}^n})$. 
\end{theorem}
\begin{proof}
For $d=100$, let $D$ be the positive integer and $M$ be the clopen subset of $F(2^{\mathbb{Z}^n})$ obtained from Theorem~\ref{thm:main}.
For any point $x\in F(2^{\mathbb{Z}^n})\backslash M$ and any $1\leq i\leq n$, let $a_i^x\leq D$ be the least non-negative integer such that $a_i^xe_i\cdot x\in M$, and let $b_i^x\leq D$ be the least non-negative integer such that $-b_i^xe_i\cdot x\in M$.

We now define a continuous edge $(2n+1)$-coloring $c$ for $F(2^{\mathbb{Z}^n})$. We will use colors $1, \dots, 2n+1$.  Consider any edge $\{x,y\}$ parallel to $e_i$. For definiteness, assume $e_i\cdot x=y$.  We consider the following cases.
\begin{enumerate}
\item[Case 1:] $a_i^x+b_i^x$ is even or $a_i^x+b_i^x$ is odd but $a_i^x>10$. In this case we define 
$$c(\{x,y\})=\left\{\begin{array}{ll}i, & \mbox{ if $b_i^x$ is even, } \\ n+i, & \mbox{ if $b_i^x$ is odd;}\end{array}\right. $$
\item[Case 2:] $a_i^x+b_i^x$ is odd and $a_i^x=10$. In this case we define $c(\{x,y\})=2n+1$;
\item[Case 3:] $a_i^x+b_i^x$ is odd and $a_i^x<10$. In this case we define 
$$c(\{x,y\})=\left\{\begin{array}{ll} n+i, & \mbox{ if $b_i^x$ is even, } \\ i, & \mbox{ if $b_i^x$ is odd.} \end{array}\right. $$
\end{enumerate}
Since color $2n+1$ only occurs in places far apart and the other colors are only correlated to the direction of the edges, we conclude that $c$ is proper. 
\end{proof}

The second application is a construction of clopen tree sections in $F(2^{\mathbb{Z}^n})$. Before presenting the theorem, we recall some definitions. A {\em tree section} in $F(2^{\mathbb{Z}^n})$ is a subgraph $G$ of the Schreier graph of $F(2^{\mathbb{Z}^n})$ where each connected component of $G$ is an infinite acyclic graph. A tree section $T$ is {\em complete} if for every $x\in F(2^{\mathbb{Z}^n})$, $T\cap [x]\neq\varnothing$, and $T$ is {\em co-complete} if for every $x\in F(2^{\mathbb{Z}^n})$, $[x]\setminus T\neq\varnothing$.

A related notion is that of line sections (see \cite[Section 1.6]{GJKS23}).  A {\em line section} in $F(2^{\mathbb{Z}^n})$ is a subgraph $L$ of the Schreier graph of $F(2^{\mathbb{Z}^n})$ where every vertex in $L$ has degree $2$. It is clear that a line section is a tree section. A line section $L$ is {\em straight} if there is $1\leq i\leq n$ such that all edges in $L$ are parallel to $e_i$. By the method of proof for \cite[Theorem 1.6.1]{GJKS23}, one can show that there do not exist  Borel straight line sections in $F(2^{\mathbb{Z}^n})$ that are both complete and co-complete. However, it was shown in \cite{GJKS24} that there are Borel line sections in $F(2^{\mathbb{Z}^2})$ where each orbit is a connected component. 

For any positive integer $k$, a {\em $k$-line section} $L$ is a line section where there are exactly $k$ many connected components of $L$ in each orbit. Similiarly define $\leq k$-line sections, $k$-tree sections, and $\leq k$-tree sections.  It was shown in \cite{GJKS23} that there are no clopen $1$-line sections in $F(2^{\mathbb{Z}^2})$. This was generalized in \cite{JO25}, where the authors show that for any positive integer $k$, there are no clopen $\leq k$-tree sections in $F(2^{\mathbb{Z}^2})$. 

In the following we show that there is a clopen tree section that is both complete and co-complete. 

\begin{theorem}\label{thm:tree}
For any integer $n\geq 2$, there is a clopen tree section in $F(2^{\mathbb{Z}^n})$ that is both complete and co-complete.
\end{theorem}

\begin{proof} Let $d_0=10$. By Theorem~\ref{thm:main} there is $D>10$ and a clopen subset $M\subseteq F(2^{\mathbb{Z}^n})$ such that
\begin{enumerate}
\item[(1)] for any $x,y\in M$, $\rho(x,y)\geq d_0=10$;
\item[(2)] for any $x\in F(2^{\mathbb{Z}^n}) $ and any $1\leq i\leq n$, there are non-negative integers $a,b\leq D$ such that $ae_i\cdot x\in M$ and $-be_i\cdot x\in M$.
\end{enumerate}

For each $x\in M$, we let $H_x$ be a finite line graph consisting of edges linking the consecutive vertices of the sequence
$$ -e_2\cdot x, \ -(e_1+e_2)\cdot x,\  -e_1\cdot x, \ (-e_1+e_2)\cdot x,\  e_2\cdot x. $$
By (1), if $x, y\in M$, then $\rho(H_x, H_y)\geq 8$. For each $x\in M$, also let $L_x$ 
be a finite line graph consisting of edges linking the consecutive vertices of the sequence
$$ e_2\cdot x,\ 2e_2\cdot x, \dots, \ ke_2\cdot x $$
where $k$ is the least positive integer so that $ke_2\cdot x\in H_y$ for some $y\in M$. By (1), $k\geq 10$. By (2), $k\leq D$. 
Then define $T=\bigcup\{H_x\cup L_x\colon x\in M\}$. Figure~\ref{fig:tree} illustrates the construction of $T$.

\begin{figure}[h]
        \centering
        \begin{tikzpicture}[scale=0.3]
\draw(0,0)to(-1,0) to (-1,2) to (0,2);

\draw(0,2)to(0,12) to (-1, 12) to (-1,14) to (0, 14);

\draw(-1,5) to (-2,5) to (-2, 7) to (-1, 7) to (-1, 12);

\coordinate (P) at (0,1);
  \fill (P) circle (4pt);
\coordinate (Q) at (0,13);
  \fill (Q) circle (4pt);
\coordinate (R) at (-1, 6);
\fill (R) circle (4pt);
        \end{tikzpicture}
\caption{\label{fig:tree}A tree section $T$ in $F(2^{\mathbb{Z}^n})$.}
\end{figure}
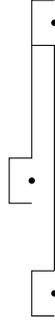

$T$ is obviously both complete and co-complete. We verify that $T$ is a tree section. For $x, y\in M$, we say that $y$ is the {\em parent} of $x$ if $L_x\cap H_y\neq \varnothing$. By our construction, every $x\in M$ has a unique parent $y\in M$, and it is clear that for the unique $g$ such that $g\cdot x=y$, $\pi_2(g)>0$ (in fact $\pi_2(g)\geq 10$). This implies that each connected component of $T$ is infinite. To see that $T$ is indeed acyclic, we consider the graph $S$ on $M$ defined by $\{x,y\}\in E(S)$ iff either $y$ is the parent of $x$ or $x$ is the parent of $y$. Then $T$ has a cycle iff $S$ has a cycle. However, if $S$ has a cycle, then there would be some $x\in M$ with two distinct parents, a contradiction. Hence $T$ is acyclic.
\end{proof}

\section{Strong markers for more general generating sets\label{sec:6}}
In this section we construct strong marker sets for more general sets of generators. In the case of dimension 2, the construction gives strong marker sets for arbitrary sets of generators.

\subsection{General packages in $\mathbb{Z}^n$}{\ }

In this subsection we develop packages for multiples of standard generators. The following lemma is a generalization of Lemma~\ref{lem:1direction}.

\begin{lemma}\label{lem:1directiongen}
Let $n, d\geq 1$ be positive integers, let $\alpha$ be a non-zero integer, and let $1\leq i\leq n$. Then there is an integer $D(n, d, \alpha)$ such that for any generalized $n$-dimensional rectangle $R$ with side length at least $D(n,d,\alpha)$ in the direction $e_i$, there is a subset $M\subseteq R$ such that
\begin{enumerate}
\item[(1)] for any distinct $x,y\in M$, $\rho(x,y)\geq d$;
\item[(2)] for any $x\in R$, there is $a\in\mathbb{Z}$ such that $x+a\alpha e_i\in M$.
\end{enumerate}
\end{lemma}

\begin{proof} Let $D(n,d,\alpha)=2|\alpha|d^n-d+(|\alpha|-1)^2$. Without loss of generality, we may assume that $i=1$ and 
$$R=\left[0,D(n,d,\alpha)-1\right]\times Q$$ 
for some generalized $(n-1)$-dimensional rectangle $Q$. Let $c\in [0, |\alpha|-1]$ be the unique integer such that 
$$ 2d^n+c\equiv 1\ (\!\!\!\!\!\!\mod |\alpha|). $$
Let 
$$ R_0=[0, 2d^n-d-1]\times Q. $$
By Lemma~\ref{lem:1direction}, there is a subset $M_0\subseteq R_0$ such that 
\begin{enumerate}
\item[(i)] for any distinct $x, y\in M_0$, $\rho(x,y)\geq d$;
\item[(ii)] for any $x\in R$, there is $a_x\in \mathbb{Z}$ such that $x+a_xe_i\in M_0$.
\end{enumerate}
Now define
$$ M=\bigcup_{0\leq t\leq |\alpha|-1} M_0+t(2d^n+c)e_1. $$
Then $M$ satisfies (1). For (2), note that for $0\leq t\leq |\alpha|-1$, 
$$ t(2d^n+c)\equiv t\ (\!\!\!\!\!\!\mod |\alpha|). $$
Thus, there is some $0\leq t\leq |\alpha|-1$ such that $a_x+t(2d^n+c)\equiv 0\ (\!\!\!\!\mod |\alpha|)$. Hence there is $a\in \mathbb{Z}$ such that $x+a\alpha e_i\in M$. To see that $M\subseteq R$, note that
$$\begin{array}{rl} \pi_1(M)\!\!\!\!&\subseteq [0, 2d^n-d-1+(|\alpha|-1)(2d^n+c)]\\
&\subseteq [0, 2d^n-d-1+(|\alpha|-1)(2d^n+|\alpha|-1)]=[0, D(n,d,\alpha)-1]. 
\end{array}$$
\end{proof}

\begin{lemma}\label{lem:1directiongenmultiple}
Let $n, d\geq 1$ be positive integers, let $\alpha_1, \dots, \alpha_m$ be non-zero integers, and let $1\leq i\leq n$. Then there is an integer $D(n, d, \alpha_1,\dots, \alpha_m)$ such that for any generalized $n$-dimensional rectangle $R$ with side length at least $D(n,d,\alpha_1, \dots, \alpha_m)$ in the direction $e_i$, there is a subset $M\subseteq R$ such that
\begin{enumerate}
\item[(1)] for any distinct $x,y\in M$, $\rho(x,y)\geq d$;
\item[(2)] for any $x\in R$ and any $1\leq j\leq m$, there is $a\in\mathbb{Z}$ such that $x+a\alpha_j e_i\in M$.
\end{enumerate}
\end{lemma}

\begin{proof} For any non-zero integer $\alpha$, let $D(n,d,\alpha)$ be the number given by Lemma~\ref{lem:1directiongen}. Let
$$ D(n,d,\alpha_1,\dots, \alpha_m)=(m-1)d+\displaystyle\sum_{j=1}^m D(n,d,\alpha_j). $$
Without loss of generality, assume $R$ is a generalized $n$-dimensional rectangle with $D(n,d,\alpha_1,\dots, \alpha_m)$ as the length of $\pi_i(R)$. Write $\pi_i(R)$ into consecutive disjoint intervals
$$ \pi_i(R)=I_1\cup J_1\cup I_2\cup J_2\cup\cdots\cup J_{m-1}\cup I_m, $$
where the length of each $J_1, \dots, J_{m-1}$ is $d-2$ (in the case $d=1$, each $J_1,\dots, J_{m-1}$ is empty) and for any $1\leq j\leq m$, the length of $I_j$ is exactly $D(n,d,\alpha_j)$. Thus we obtain disjoint subrectangles of $R$
$$ R_1,\dots, R_m $$
with $\rho(R_j, R_k)\geq d$ for any $1\leq j\neq k\leq m$, and for $1\leq j\leq m$, the length of $\pi_i(R_j)$ is exactly $D(n,d,\alpha_j)$. Now (1) and (2) follow from Lemma~\ref{lem:1directiongen}.
\end{proof}

Lemma~\ref{lem:1directiongenmultiple} can be used in conjunction with Lemma~\ref{lem:1directionmultiple} to give proofs of the following analogs of Theorems~\ref{thm:strongZn} and \ref{thm:main}.

\begin{theorem}\label{thm:strongZngen} Let $n, d_0, m_1, m_2, \dots, m_n\geq 1$ be positive integers. Let 
$$\alpha_{1, 1}, \dots, \alpha_{1, m_1}, \alpha_{2,1},\dots, \alpha_{2,m_2}, \dots, \alpha_{n,1},\dots, \alpha_{n, m_n}$$ be non-zero integers. Then there is a positive integer $D_0$ such that for any $n$-dimensional rectangle $R$ with side length at least $D_0$ in each direction, there is a subset $M\subseteq R$ such that
\begin{enumerate}
\item[(1)] for any distinct $x,y\in M$, $\rho(x,y)\geq d_0$;
\item[(2)] for any $1\leq i\leq n$, $1\leq j\leq m_i$ and $x\in R$, there is $a_{i,j}\in\mathbb{Z}$ such that $x+a_{i,j}\alpha_{i,j}e_i\in M$.
\end{enumerate}
\end{theorem}

\begin{theorem}\label{thm:strongZngen} Let $n, d_0, m_1, m_2, \dots, m_n\geq 1$ be positive integers. Let 
$$\alpha_{1, 1}, \dots, \alpha_{1, m_1}, \alpha_{2,1},\dots, \alpha_{2,m_2}, \dots, \alpha_{n,1},\dots, \alpha_{n, m_n}$$ be non-zero integers. Then there is a positive integer $D_1$ and a clopen subset $M\subseteq F(2^{\mathbb{Z}^n})$ such that 
\begin{enumerate}
\item[(1)] for any distinct $x,y\in M$ in the same orbit, $\rho(x,y)\geq d_0$;
\item[(2)] for any $1\leq i\leq n$, $1\leq j\leq m_i$ and $x\in F(2^{\mathbb{Z}^n})$, there are non-negative integers $a_{i,j}, b_{i,j}\leq D_1$ such that $a_{i,j}\alpha_{i,j}e_i\cdot x\in M$ and $-b_{i,j}\alpha_{i,j}e_i\cdot x\in M$.
\end{enumerate}
\end{theorem}

\subsection{Parallelopipeds as packages in $\mathbb{Z}^n$}{\ }

In this subsection we develop packages for generators whose coordinates are non-zero. For such generators, we will use packages that are generalized parallelopipeds. These are best described in the context of $\mathbb{R}^n$. We need to make some definitions. Let $v=(\nu_1,\dots, \nu_n)\in \mathbb{R}^n$ be a vector with $\nu_i\neq0$ for all $1\leq i\leq n$. Fix an $1\leq i\leq n$. Let $B$ be a generalized $n$-dimensional rectangle in $\mathbb{R}^n$ with $\pi_i(B)$ a singleton. Let $h>0$ be a real number. Then the {\em generalized parallelopiped} in $\mathbb{R}^n$ with {\em base} $B$, {\em direction $v$} and {\em $e_i$-height} $h$  is the set 
$$ P(B, v, h)=\bigcup_{0\leq t\leq h/|\nu_i|} B+tv. $$
We also consider the following {\em infinite generalized parallelopiped} in $\mathbb{R}^n$ with {\em base} $B$ and {\em direction} $v$
$$ P(B, v)=B+\mathbb{R}v. $$
Now let $S$ be a generalized $n$-dimensional rectangle in $\mathbb{Z}^n$ with $\pi_i(S)$ a singleton. Let $\mbox{Hull}(S)$ be the convex hull of $S$ in $\mathbb{R}^n$. If $v=(\nu_1,\dots, \nu_n)\in\mathbb{Z}^n$ is a vector with $\nu_i\neq 0$ for all $1\leq i\leq n$, and $h$ is a positive integer, then we let
$$ L(S, v, h)=P(\mbox{Hull}(S), v, h)\cap \mathbb{Z}^n $$
and call it the {\em generalized parallelopiped} in $\mathbb{Z}^n$ with {\em base} $S$, {\em direction $v$}, and {\em $e_i$-height} $h$. The {\em infinite generalized parallelopiped} with {\em base} $S$ and {\em direction} $v$ is 
$$ L(S, v)=P(\mbox{Hull}(S), v)\cap \mathbb{Z}^n. $$

\begin{lemma}\label{lem:1directionslanted}
Let $n, d\geq 1$ be positive integers, let $v=(\nu_1,\dots, \nu_n)\in\mathbb{Z}^n$ have only non-zero coordinates, let $S$ be a generalized $n$-dimensional rectangle with $\pi_i(S)$ a singleton, and let $1\leq i\leq n$. Then there is an integer $H(n, d, v, i)$ such that for any generalized parallelopiped $R$ with base $S$, direction $v$ and $e_i$-height at least $H(n,d,v,i)$, there is a subset $M\subseteq R$ such that
\begin{enumerate}
\item[(1)] for any distinct $x,y\in M$, $\rho(x,y)\geq d$;
\item[(2)] for any $x\in L(S, v)$, there is $a\in\mathbb{Z}$ such that $x+av\in M$.
\end{enumerate}
\end{lemma}

\begin{proof} Without loss of generality assume that $i=n$ and 
$$ S=[a_1,b_1]\times  \cdots \times [a_{n-1},b_{n-1}]\times\{0\}=Q\times \{0\}. $$
By straightforward computation, we can express $L(S,v)$ as
$$ \bigcup\{S_t+\mathbb{Z}v\colon 0\leq t\leq \nu_n-1\} $$
where, for each $0\leq t\leq \nu_n-1$,
$$ S_t=[a_{1,t},b_{1,t}]\times\cdots \times[a_{n-1,t}, b_{n-1,t}]\times \{t\} $$
and
$$ [a_{i,t}, b_{i,t}]=\left\{ x\in\mathbb{Z}\colon a_i+t\displaystyle\frac{\nu_i}{\nu_n}\leq x\leq b_i+t\frac{\nu_i}{\nu_n}\right\} $$
for each $1\leq i\leq n-1$.

Let 
$$ \alpha=\mbox{gcd}(\nu_1,\dots, \nu_n) $$
and let $D(n,d,\alpha)$ be the number given by Lemma~\ref{lem:1directiongen}. Inductively define
$$ h_0=D(n,d,\alpha)\alpha^{-1}\nu_n,\ h_{k+1}=h_k+h_kd\|v\|(h_k+1)d\nu_n \mbox{ for $0\leq k<n-1$.} $$
Indutively define
$$ H_0=h_{n-1},\ H_{t+1}=(H_t+d)\nu_n \mbox{ for $0\leq t<\nu_n-1$.} $$
Let
$$ H(n,d,v,i)=H_{\nu_n-1}. $$
Now let $R$ be a generalized parallelopiped with base $S$, direction $v$ and $e_i$-height $H(n,d,v,i)$. We construct sets $M_t\subseteq R$ for $t\in [0, \nu_n-1]$ then let 
$M=\bigcup_{t\in [0, \nu_n-1]}M_t$. 

To construct $M_0$, we use induction on the ``dimension" of $Q$, i.e.,
$$ \mbox{dim}(Q)=|\{1\leq i\leq n-1\colon a_i\neq b_i\}|. $$
First assume $\mbox{dim}(Q)=0$. In this case, $S$ is a singleton $\{x_0\}$.  Applying Lemma~\ref{lem:1directiongen} to the ``$1$-dimensional parallelopiped" in $\mathbb{Z}^n$
$$ R_0=x_0+[0, D(n,d,\alpha)-1](\alpha^{-1}v)\subseteq R, $$
we obtain a subset $M_{0,0}\subseteq R_0$ such that
\begin{enumerate}
\item[(i)] for any distinct $x, y\in M_{0,0}$, $\rho(x,y)\geq d \|\alpha^{-1}v\|\geq d$;
\item[(ii)] for any $x\in L(\{x_0\}, v)=(x_0+\mathbb{R}v)\cap\mathbb{Z}^n$, there is $a\in \mathbb{Z}$ such that $x+a\alpha^{-1}v\in M_{0,0}$.
\end{enumerate}
Note that $R_{0,0}$ is a generalized parallelopiped with base $\{x_0\}$, direction $v$ and $e_n$-height $D(n,d,\alpha)\alpha^{-1}\nu_n=h_0$.

In general, assume $\dim(Q)=k+1$ where $0\leq k< n-1$. Without loss of generality, assume
$$ Q=C\times [a_{n-1},b_{n-1}] $$
where $a_{n-1}<b_{n-1}$ and $\mbox{dim}(C)=k$. Moreover, assume $a_{n-1}=0$ and $b_{n-1}=\ell>0$. For each $0\leq \lambda\leq \ell$, let
$$ Q_\lambda=C\times \{\lambda\}. $$
Then
$$ Q_\lambda=Q_0+\lambda e_{n-1} \mbox{ and } Q=\bigcup_{\lambda\in [0,\ell]}Q_\lambda. $$
By the inductive hypothesis, for any generalized parallelopiped $R_{0,k}$ with base $Q_0\times\{0\}$, direction $v$ and $e_n$-height $h_k$, there is a subset $M_{0,k}\subseteq R_{0,k}$ such that (1) and (2) hold. Let $m_k=h_kd\|v\|$. Let $R_{0,k+1}$ be a generalized parallelopiped with base $Q\times \{0\}$, direction $v$ and $e_n$-height $h_{k+1}$. Define
$$ M_{0,k+1}=\bigcup_{0\leq\lambda\leq\ell} (M_{0,k}+\lambda e_{n-1})+(\lambda\!\!\!\! \mod m_k) (h_k+1)dv. $$
Then $M_{0,k+1}$ has height $h_k+m_k(h_k+1)d\nu_n=h_{k+1}$, and satisfies (1) and (2) with base $S$. 

Now let $M_0=M_{0,\delta}$, where $\delta=\mbox{dim}(Q)\leq n-1$. Then $M_0\subseteq R_0$ has height at most $H_0$. 

Finally, for each $0\leq t\leq \nu_n-1$, consider $S_t$. Note that for each $1\leq i\leq n-1$, $b_{i,t}-a_{i,t}=b_i-a_i$. Thus $S_t$ is a ``copy" of $S_0=S$ except that its $e_n$-coordinate is $t$. Thus we can define
$$ M_t=M_0+(a_{1,t},\dots, a_{n-1,t},t)-(a_1,\dots, a_n,0)+(H_t+d)v $$
and $M=\bigcup_{0\leq t\leq \nu_n-1}M_t$. Then $M$ is as required.
\end{proof}

Let $R_0, R_1$ be $n$-dimensional rectangles in $\mathbb{Z}^n$ and let $\delta>0$ be a positive integer. We say that $R_1$ is the  {\em $\delta$-extension} of $R_0$, or $R_0$ is the {\em $\delta$-core} of $R_1$, if 
$$ R_0=\left\{ x\in R_1\colon \rho(x, R_1^c)>\delta\right\}, $$
where $R_1^c=\mathbb{Z}^n\setminus R_1$ is the complement of $R_1$. Note that if the side lengths of $R_1$ are at least $2\delta+1$ then its $\delta$-core is an $n$-dimensional rectangle. Conversely, any $n$-dimensional rectangle has a $\delta$-extension. 

\begin{lemma}\label{lem:1directionslanted2}
Let $n, d\geq 1$ be positive integers and let $v=(\nu_1,\dots, \nu_n)\in\mathbb{Z}^n$ have only non-zero coordinates. Then there is an integer $\delta(n, d, v)$ such that for any $n$-dimensional rectangle $R_0$, there is a subset $M\subseteq R_1\setminus R_0$, where $R_1$ is the $\delta(n,d,v)$-extension of $R_0$, such that
\begin{enumerate}
\item[(1)] for any distinct $x,y\in M$, $\rho(x,y)\geq d$ and $\rho(x, R_1^c)\geq d$;
\item[(2)] for any $x\in R_0$, there is $a\in\mathbb{Z}$ such that $x+av\in M$.
\end{enumerate}
\end{lemma}

\begin{proof} Let $H(n,d, v, i)$ be the numbers given by Lemma~\ref{lem:1directionslanted}. Let 
$$\delta(n,d,v)=(n+1)d\|v\|+\displaystyle\sum_{i=1}^n H(n,d,v,i)\|v\|. $$
Let $R_0$ be an $n$-dimensional rectangle. Let $R_1$ be the $\delta(n,d,v)$-extension of $R_0$. Let 
$$ k=\displaystyle\sum_{i=1}^n \alpha_i 2^{i-1}, $$
where 
$$ \alpha_i=\left\{\begin{array}{ll} 0, & \mbox{ if $\nu_i<0$,} \\ 1, & \mbox{ if $\nu_i>0$.}\end{array}\right. $$
Let $x_k$ be the $k$-th element of the canonical enumeration of all extreme points of $R_0$. Note that $x_k$ is an element of $n$ many faces of $R_0$. According to the normal vector of these faces, we enumerate them as $S_1, \dots, S_n$. For $1\leq i\leq n$, $S_i$ is perpendicular to $e_i$. Also, by our choice of $x_k$, we have that for any $1\leq i\leq n$ and any positive integer $a$, $(S_i+av)\cap R_0=\varnothing$. 

Define inductively
$$ v_1=v, \ v_{i+1}=v_i+(H(n,d,v, i)+d)v \mbox{ for $1\leq i<n$.} $$
For each $1\leq i\leq n$, apply Lemma~\ref{lem:1directionslanted} to the parallelopiped $R_i$ with base $S_i+v_i$, direction $v$ and $e_i$-height $H(n,d,v,i)$ to obtain a subset $M_i\subseteq R_i$ such that
\begin{enumerate}
\item[(1)] for any distinct $x, y\in M_i$, $\rho(x,y)\geq d$;
\item[(2)] for any $x\in L(S_i, v)$, there is $a\in \mathbb{Z}$ such that $x+av\in M_i$.
\end{enumerate}
For dimension 2, these parallelopipeds are illustrated in Figure~\ref{fig:parallelogram}.

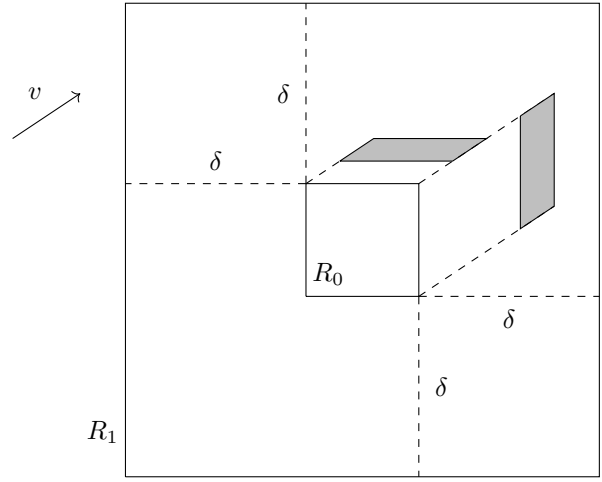
\begin{figure}[h]
        \centering
        \begin{tikzpicture}[scale=0.3]

\draw (0,0) to (21,0) to (21,21) to (0,21) to (0,0);
\draw(8,8)to(8,13) to (13,13) to (13,8) to (8,8);
\draw[->] (-5,15)to(-2,17);
\node at (-4, 17) {$v$};
\node at (-1, 2) {$R_1$};
\node at (9, 9) {$R_0$};

\draw[dashed] (8,13) to (8,21);
\draw[dashed] (8,13) to (0,13);
\draw[dashed] (13,8) to (21,8);
\draw[dashed] (13,8) to (13,0);
\node at (7,17) {$\delta$};
\node at (4, 14) {$\delta$};
\node at (17, 7) {$\delta$};
\node at (14, 4) {$\delta$};

\draw[fill=lightgray] (9.5,14) to (14.5,14) to (16, 15) to (11,15) to (9.5,14);
\draw[dashed] (13,13) to (19, 17);
\draw[fill=lightgray] (17.5,16) to (19, 17) to (19, 12) to (17.5, 11) to (17.5,16);
\draw[dashed] (13,8) to (19,12);
\draw[dashed] (8,13) to (9.5,14);



        \end{tikzpicture}
\caption{\label{fig:parallelogram}Using parallelograms as packages in $\mathbb{Z}^2$.}
\end{figure}

Let $M=\bigcup_{i=1}^n M_i$. By our construction, for any distinct $1\leq i<j\leq n$, $\rho(R_i, R_j)\geq d\|v\|\geq d$; hence $\rho(M_i, M_j)\geq d$. This guarantees that for any distinct $x, y\in M$, $\rho(x, y)\geq d$. Also, note that 
$$\delta(n,d,v)-\|v_n\|-H(n,d,v,n)\|v\|\geq d\|v\|\geq d, $$
hence $\rho(M, R_1^c)\geq d$.

Finally, since $R_0\subseteq \bigcup_{i=1}^n L(S_i, v)$, we have that for any $x\in R_0$, there is $a\in \mathbb{Z}$ such that $x+av\in M$.
\end{proof}

\subsection{Clopen strong markers}{\ }

We are now ready to prove our second main theorem of the paper.

\begin{theorem}\label{thm:maingen} Let $n, d_0\geq 1$ be positive integers and let $S\subseteq \mathbb{Z}^n$ be a finite generating set. Suppose for each $g\in S$, $\{1\leq i\leq n\colon \pi_i(g)\neq 0\}$ has size either $1$ or $n$. Then there is a positive integer $\Delta$ and a clopen subset $M\subseteq F(2^{\mathbb{Z}^n})$ such that
\begin{enumerate}
\item[(1)] for any distinct $x,y\in M$ in the same orbit, $\rho(x,y)\geq d_0$;
\item[(2)] for any $g\in S$ and any $x\in F(2^{\mathbb{Z}^n})$, there are non-negative integers $a, b\leq \Delta$ such that $ag\cdot x\in M$ and $-bg\cdot x\in M$.
\end{enumerate}
\end{theorem}

The rest of this subsection is devoted to a proof of Theorem~\ref{thm:maingen}. Let $n, d_0$ and $S$ be given. We first apply Lemma~\ref{lem:basicmarker} with an unspecified even number $\Delta>4$ to obtain a clopen equivalence relation $E_{\Delta}$ on $F(2^{\mathbb{Z}^n})$ such that each of the $E_\Delta$-equivalence classes is an $n$-dimensional rectangle with side lengths either $\Delta$ or $\Delta+1$. For each $x\in F(2^{\mathbb{Z}^n})$, let $[x]_E$ be the $E_\Delta$-equivalence class containing $x$. 

Let $\delta<\frac{1}{2}\Delta-1$ be a positive integer. Let
$$ C_\delta=\{x\in F(2^{\mathbb{Z}^n})\colon \rho(x, [x]\setminus [x]_E)> \delta\}. $$
Then $C_\delta$ consists of regions that are $n$-dimensional rectangles in $F(2^{\mathbb{Z}^n})$; in fact, for each $E_\Delta$-class $R$, $C_\delta\cap R$ is the $\delta$-core of $R$, which is an $n$-dimensional rectangle with side lengths $\Delta-2\delta$ or $\Delta-2\delta+1$.

\begin{lemma}\label{lem:hit} Let $v=(\nu_1,\dots, \nu_n)\in\mathbb{Z}^n$ where $\nu_i\neq 0$ for all $1\leq i\leq n$. Suppose 
$\Delta>2^{n+1}(\delta+1)\|v\|$. Then for any $x\in F(2^{\mathbb{Z}^n})$, there is an integer $a\in [-\frac{1}{2}\Delta, \frac{1}{2}\Delta]$ such that $av\cdot x\in C_\delta$.
\end{lemma}

\begin{proof} First note that for any $E_\Delta$-class $R$, if $[-\delta, \delta]v\cdot x\subseteq R$, then $x\in C_\delta\cap R$. 
Suppose $x\in F(2^{\mathbb{Z}^n})$ is arbitrary. Consider 
$$ T=\{y\in F(2^{\mathbb{Z}^n})\colon \rho(x,y)\leq \frac{1}{2}\Delta\}. $$
Then $T$ is an $n$-dimensional rectangle with side length $\Delta$ in each direction. $T$ has $x$ at its center, and since $\Delta>2^{n+1}(\delta+1)\|v\|$, the interval
$$ \{a\in\mathbb{Z}\colon  av\cdot x\in T\}$$ has length at least $2^{n+1}(\delta+1)$. Note that $T$ meets at most $2^n$ many $E_\Delta$-classes $R$. It follows that for some such $R$, there is an interval $I$ of length $2\delta$ such that $Iv\cdot x\subseteq R\cap T$, and in particular there is $a\in [-\frac{1}{2}\Delta, \frac{1}{2}\Delta]$ such that $av\cdot x\in C_\delta$. 
\end{proof}

Now enumerate $S$ as
$$v_1,\ \dots,\ v_{m_0},\ \alpha_{1,1}e_1,\ \dots,\ \alpha_{1, m_1}e_1,\ \dots,\ \alpha_{n,1}e_n,\ \dots,\ \alpha_{n, m_n}e_n, $$
for natural numbers $m_0, m_1,\dots, m_n$, integers $\alpha_{i,j}$ for $1\leq i\leq n$ and $1\leq j\leq m_i$, where $v_1, \dots, v_{m_0}$ enumerate all elements of $S$ whose coordinates are all non-zero. Without loss of generality assume $m_0, m_1, \dots, m_n$ are positive.

We now fix the parameters. Let $D_1$ be given by Theorem~\ref{thm:strongZngen}. Let $\mu_0=2D_1$. For each $1\leq k\leq m_0$, inductively define
$$  \mu_k=\mu_{k-1}+\delta(n,d_0,v_k), $$
where $\delta(n, d_0, v_k)$ is given by Lemma~\ref{lem:1directionslanted2}. Finally, let 
$$\Delta>2^{n+1}(\mu_{m_0}+1)(\|v_1\|+\dots+\|v_{m_0}\|)$$ be an even multiple of $D_1$.

We are now ready to define the clopen marker set $M$. First we define a refinement of $E_\Delta$. This is done by dividing each side of an $E_\Delta$-class into intervals of length $D_1$ or $D_1+1$ (with at most one such interval of length $D_1+1$), which gives rise to a division of each $E_\Delta$-class into $n$-dimensional rectangles with side lengths either $D_1$ or $D_1+1$. Note that this is possible since $\Delta$ is a multiple of $D_1$. We denote the resulting equivalence relation $F$. Note that $F$ is clopen and each $F$-class is an $n$-dimensional rectangle with side length either $D_1$ or $D_1+1$. Note that Theorem~\ref{thm:strongZngen} can be proved for this setup, and we obtain a clopen subset $\tilde{M}\subseteq F(2^{\mathbb{Z}^n})$ such that
\begin{enumerate}
\item[(1)] for any distinct $x,y\in \tilde{M}$ in the same orbit, $\rho(x,y)\geq d_0$;
\item[(2)] for any $1\leq i\leq n$, $1\leq j\leq m_i$ and $x\in F(2^{\mathbb{Z}^n})$, there are non-negative integers $a_{i,j}, b_{i,j}\leq D_1$ such that $a_{i,j}\alpha_{i,j}e_i\cdot x\in \tilde{M}$ and $-b_{i,j}\alpha_{i,j}e_i\cdot x\in \tilde{M}$.
\end{enumerate}

Before we continue the definition of $M$ we need the following definition. For each $n$-dimensional rectangle $R$ and for any $1\leq i\leq n$, let $F_i(R)=\{x\in R\colon e_i\cdot x\not\in R\}$ and refer to it as the {\em upper face} of $R$ in the direction $e_i$. 

Let $X$ be the union of all $F$-classes $U$ so that for some $E_\Delta$-class $R$ and $1\leq i\leq n$, $U\cap F_i(R)\neq\varnothing$. $X$ is a clopen subset of $F(2^{\mathbb{Z}^n})$. Let $M_0=\tilde{M}\cap X$. Then we have
\begin{enumerate}
\item[(3)] for any distinct $x,y\in M_0$ in the same orbit, $\rho(x,y)\geq d_0$;
\item[(4)] for any $1\leq i\leq n$, $1\leq j\leq m_i$ and $x\in F(2^{\mathbb{Z}^n})$, there are non-negative integers $a_{i,j}, b_{i,j}\leq 2\Delta+1$ such that $a_{i,j}\alpha_{i,j}e_i\cdot x\in M_0$ and $-b_{i,j}\alpha_{i,j}e_i\cdot x\in M_0$.
\end{enumerate}
Here (3) follows from (1). For (4), note that for any $x\in F(2^{\mathbb{Z}^n})$ and $1\leq i\leq n$, there are non-negative integers $a, b\leq \Delta+1$ such that $ae_i\cdot x\in X$ and $-be_i\cdot x\in X$. Now (4) follows from (2). 

For each $0\leq k\leq m_0$, let 
$$ K_k=C_{\mu_k} $$
where $C_\delta$ is as defined before Lemma~\ref{lem:hit}. Then for any $E_\Delta$-class $R$,
$$ K_0\cap R\supseteq K_1\cap R\supseteq K_2\cap R\supseteq \cdots \supseteq K_{m_0}\cap R. $$
For $1\leq k\leq m_0$, since $\mu_{k}-\mu_{k-1}=\delta(n,d_0, v_k)$, each $K_{k}\cap R$ is the $\delta(n,d_0,v_k)$-core of $K_{k-1}\cap R$. Now for each $1\leq k\leq m_0$ we apply Lemma~\ref{lem:1directionslanted2} with $n, d_0, v_k$ to $K_k$ to obtain a subset $M_k\subseteq K_{k-1}\setminus K_k$ such that
\begin{enumerate}
\item[(5)] for any distinct $x, y\in M_k$, $\rho(x,y)\geq d$ and $\rho(x, [x]\setminus K_{k-1})\geq d$;
\item[(6)] for any $E_\Delta$-class $R$ and $x\in K_k\cap R$, there is $a\in \mathbb{Z}$ such that $x+av\in M_k\cap R$.
\end{enumerate}

Now let $M=\bigcup_{0\leq k\leq m_0} M_k$. For each $1\leq k\leq m_0$, since
$$ \Delta>2^{n+1}(\mu_k+1)\|v_k\|, $$
Lemma~\ref{lem:hit} can be applied to guarantee that 
\begin{enumerate}
\item[(7)] for any $x\in F(2^{\mathbb{Z}^n})$, there is an integer $a\in [-\frac{1}{2}\Delta, \frac{1}{2}\Delta]$ such that $av_{k}\cdot x\in K_{k}$. 
\end{enumerate}
Combining (6) and (7), we get that for any $1\leq k\leq m_0$ and any $x\in F(2^{\mathbb{Z}^n})$, there are non-negative integers $a_k, b_k\leq 2\Delta+1$ such that $a_kv_k\cdot x\in M$ and $-b_kv_k\cdot x\in M$. Then by (3) and (5), we have that for any distinct $x, y\in M$, $\rho(x,y)\geq d_0$. 

This completes the proof of Theorem~\ref{thm:maingen}.

\subsection{Applications}{\ }

Let $n\geq 1$ and let $S\subseteq \mathbb{Z}^n$ be a generating set. The {\em Schreier graph} on $F(2^{\mathbb{Z}^n})$ with generating set $S$ is the graph $G$ with vertex set $F(2^{\mathbb{Z}^n})$ and the edge set defined as 
$$ \{x,y\}\in E(G)\iff \exists g\in S\ (g\cdot x=y \mbox{ or } g\cdot y=x). $$

In \cite{GWW25} it was shown that for any $n\geq 1$ and generating set $S\subseteq \mathbb{Z}^n$, there is a continuous proper edge $(2|S|+1)$-coloring of the Schreier graph on $F(2^{\mathbb{Z}^n})$ with generating set $S$. In fact, if $S$ does not contain the identity and $S$ is symmetric, i.e., $g^{-1}\in S$ whenever $g\in S$, then the continuous edge chromatic number of the Schreier graph on $F(2^{\mathbb{Z}^n})$ with generating set $S$ is exactly $2|S|+1$.

By a proof similar to that of Theorem~\ref{thm:coloring}, we have the following corollary.

\begin{corollary}[\cite{GWW25}] Let $n\geq 1$ and let $S\subseteq \mathbb{Z}^n$ be a finite generating set. Suppose for each $g\in S$, $\{1\leq i\leq n\colon \pi_i(g)\neq 0\}$ has size either $1$ or $n$. Then there is a continuous proper edge $(2|S|+1)$-coloring of the Schreier graph on $F(2^{\mathbb{Z}^n})$ with generating set $S$. 
\end{corollary}

We also have the following conclusion for dimension $2$. 

\begin{corollary}[\cite{GWW25}] Let $S\subseteq \mathbb{Z}^2$ be any generating set. Then there is a continuous proper edge $(2|S|+1)$-coloring of the Schreier graph on $F(2^{\mathbb{Z}^2})$ with generating set $S$.
\end{corollary}

\begin{proof} Note that for any $g\in \mathbb{Z}^2$ with $g\neq (0,0)$, $|\{1\leq i\leq 2\colon \pi_i(g)\neq 0\}|$ is either $1$ or $2$. Thus the condition in Theorem~\ref{thm:strongZngen} always holds.
\end{proof}

The proofs to these corollaries are different from those in \cite{GWW25}.

\end{document}